\documentclass[12pt, reqno]{amsart}
\usepackage{amsmath, amsthm, amscd, amsfonts, amssymb, graphicx, color}
\usepackage[bookmarksnumbered, colorlinks, plainpages]{hyperref}
\usepackage{tikz-cd}
\hypersetup{colorlinks=true,linkcolor=red, anchorcolor=green, citecolor=cyan, urlcolor=red, filecolor=magenta, pdftoolbar=true}

\newtheorem{theorem}{Theorem}[section]
\newtheorem{lemma}[theorem]{Lemma}
\newtheorem{proposition}[theorem]{Proposition}
\newtheorem{corollary}[theorem]{Corollary}
\theoremstyle{definition}
\newtheorem{definition}[theorem]{Definition}

\theoremstyle{remark}
\newtheorem{remark}[theorem]{Remark}
\numberwithin{equation}{section}
\newtheorem{notation}[theorem]{Notation}

\def\C{\mathbb{C}}
\def\N{\mathbb{N}}

\def\cH{\mathcal{H}}
\def\cA{\mathcal{A}}

\def\cB{\mathcal{B}}
\def\cU{\mathcal{U}}

\def\eme{\mathcal{M}}

\def\cT{\mathcal{T}}

\def\cI{\mathcal{I}}

\providecommand{\abs}[1]{\lvert#1\rvert}
\providecommand{\Abs}[1]{\Big\lvert#1\Big\rvert}
\providecommand{\norm}[1]{\lVert#1\rVert}

\begin{document}

\setcounter{page}{1}

\title[Ultrapowers and multipliers of a $C^{\ast}$-algebra ]{An ultrapower construction of the multiplier algebra of a $C^{\ast}$-algebra 
and an application to boundary amenability of groups}

\author[F. Poggi, R. Sasyk]{Facundo Poggi $^1$, \MakeLowercase{and} Rom\'an Sasyk $^{1,2}$}

\address{$^{1}$Departamento de Matem\'atica, Facultad de Ciencias Exactas y Naturales, Universidad de Buenos Aires, Argentina.}

\address{$^{2}$Instituto Argentino de Matem\'aticas-CONICET
Saavedra 15, Piso 3 (1083), Buenos Aires, Argentina;
}
\email{\textcolor[rgb]{0.00,0.00,0.84}{fpoggi@dm.uba.ar}}

\email{\textcolor[rgb]{0.00,0.00,0.84}{rsasyk@dm.uba.ar}}

\subjclass[2010]{Primary 46L05; Secondary 03C20, 20F65.}

\keywords{Multiplier algebra, ultraproduct of $C^{\ast}$-algebras, boundary amenable groups}

\begin{abstract}
Using ultrapowers of $C^{\ast}$-algebras we provide a new construction of the multiplier algebra of a $C^{\ast}$-algebra. 
This extends the work of Avsec and Goldbring [Houston J. Math., to appear, arXiv:1610.09276.] 
to the setting of noncommutative and nonseparable $C^{\ast}$-algebras.
We also extend their work
 to give a new proof of the fact that groups that act transitively on locally finite trees with boundary amenable stabilizers are 
boundary amenable.
\end{abstract} \maketitle

\section{introduction}

The multiplier algebra $\eme(\cA)$ of a $C^{\ast}$-algebra $\cA$ is a $C^{\ast}$-algebra that contains $\cA$ as an essential ideal and satisfies the following universal property: for every $C^{\ast}$-algebra $\cB$ containing $\cA$ as an ideal, there exists a unique $\ast$-homomorphism $\varphi: \cB \to\eme(\cA)$ such that $\varphi$ is the identity on $\cA$.
If $\cA$ is abelian, thus of the form $C_{0}(X)$ for some locally compact Hausdorff space $X$,  then $\eme(\cA)$ is isomorphic to $C_{b}(X)$ and this in turn can be identified with $C(\beta X)$, where $\beta X$ is the Stone-\u{C}ech compactification of $X$ (for more about multiplier algebras, see, for instance \cite{MR0225175, MR3839621}).
 
In the article \cite{2016arXiv161009276A}, Avsec and Goldbring provided a new construction of the multiplier algebra of the abelian $C^{\ast}$-algebra $C_{0}(X)$ using ultraproducts of $C^{\ast}$-algebras, in the case  when $X$ is a second countable locally compact Hausdorff space. From there, they inferred a new construction of the Stone-\u{C}ech compactification of $X$, and they used it to give a new proof of the fact that groups that act properly and transitively on trees are boundary amenable.
In section \ref{sec2} of this note, we extend their work providing a construction of the multiplier algebra of any $C^{\ast}$-algebra $\cA$ by means of ultraproducts of $C^{\ast}$-algebras. 
In section \ref{sec3}, we focus on the case of commutative and separable $C^{\ast}$-algebras, and compare our main technical tool with the main technical tool used in \cite{2016arXiv161009276A} to explain why the work done here is indeed a generalization of \cite{2016arXiv161009276A}.
Finally, in section \ref{sec4}, we extend the techniques of \cite{2016arXiv161009276A} to show that groups that act transitively on
 locally finite trees having boundary amenable stabilizers are boundary amenable. 

\section{Ultraproducts and Multipliers}\label{sec2}

Let $\mathcal I$ be a set. An {\it ultrafilter} over $\cI$ is a nonempty collection $\cU$ of subsets of $\cI$ with the following properties:
\begin{enumerate}
\item finite intersection property: for every $\cI_{0},\cI_{1}\in\cU$, then $\cI_{0}\cap \cI_{1} \in\cU$; 
\item  directness: for every $\cI_{o} \subset \cI_{1}$, where $\cI_{o}$ belongs to $\cU$, then $\cI_{1}\in\cU$;
\item maximality: for every $\cI_{0}\subset \cI$, either $\cI_{0} \in \cU$ or $\cI\setminus \cI_{0} \in \cU$.
\end{enumerate}

An ultrafilter is {\it principal} if there exists $i\in \cI$ such that the subsets of $\cI$ that contains $i$ are in the ultrafilter. Ultrafilters not of this form are called {\it nonprincipal} or {\it free}. It is easy to show that an ultrafilter is nonprincipal exactly when it contains no finite sets. An ultrafilter is {\it cofinal} when the index set is a directed set, and the sets $\{i\in \cI: i\geq i_{0}\} $ are in $\cU$ for every $i_{0}\in \cI$.

When dealing with directed sets with the property that there is no maximal element, every cofinal ultrafilter is nonprincipal. Moreover, when the ultrafilter is over $\N$, being cofinal is the same as being nonprincipal. If a directed set has a maximal element, then every cofinal ultrafilter is principal.

\begin{definition}\label{convergenciaultrafiltro}
Let $\cU$ be an ultrafilter over $\cI$. Let $(X,d)$ be a metric space and let $(a_{i})_{i\in \cI}\subset X$. We say that $(a_{i})_{i\in\cI}$ 
is convergent along $\cU$, if there exists an element $a\in X$ such that, for every $\varepsilon>0$, the set $\{i\in\cI: d(a_{i},a)<\varepsilon\}$ belongs to $\cU$. The element $a$ is called the $\cU$-limit of $(a_{i})_{i\in\cI}$ and it is denoted by $\displaystyle\lim_{\cU}{a_{i}}$.
\end{definition}

 \subsection{Ultraproducts of $C^{\ast}$-algebras}
 
Let $\cU$ be an ultrafilter over a set $\cI$ and let $\mathcal{A}$ be a $C^{\ast}$-algebra. Denote by $\prod_{\cI}{\mathcal{A}}$  the set $\{(a_{i})_{i\in \cI}: \sup_{i\in \cI}{\norm{a_{i}}}<\infty\}$ and let $\mathcal{N}_{\cU}$ be the subspace generated by those $(a_{i})_{i\in\cI}\in\prod_{\cI}{\mathcal{A}}$ such that $\displaystyle\lim_{\cU}{\norm{a_{i}}} =0$. Denote by $\mathcal{A}^{\cU}$ the quotient $\prod_{\cI}{\mathcal{A}}  / \mathcal{N}_{\cU}$. This is a vector space, and with the norm defined by $\norm{(a_{i})_{i\in\cI}}_{\cU}:= \displaystyle\lim_{\cU}{\norm{a_{i}}}$, and the involution defined by $(a_{i})_{i\in\cI}^{*} := (a_{i}^{*})_{i\in\cI}$, so $\cA^{\cU}$ becomes a $C^{\ast}$-algebra.

\begin{remark} \label{conjuntosomega} 
Let $((a_{i}^{n})_{i\in\cI})_{n\in\mathbb N}\subset\mathcal{A}^{\cU}$ be a sequence that converges to $(a_{i})_{i\in\cI}\in\mathcal{A}^{\cU}$. Then, for all $\varepsilon>0$, there exists $n_{0}\in\N$ such that if $n\geq n_{0}$, then ${\norm{(a^{n}_{i})_{i\in\cI}-(a_{i})_{i\in\cI}}}_{\cU}<\varepsilon$. 
 We claim that if $n\geq n_{0}$, then the set 
  $\Omega_{n}(\varepsilon) := \{i\in \cI: \norm{a_{i}^{n} - a_{i}}<\varepsilon\}$ belongs to $\cU.$
To show this, set $\alpha_{n} := {\norm{(a^{n}_{i})_{i\in\cI}-(a_{i})_{i\in\cI}}}_{\cU}$.
For every $\delta>0$, we have $\{i\in \cI:\abs{\norm{a_{i}^{n}-a_{i}} -\alpha_{n}}<\delta\}\in\cU$. 
Taking $\delta = \varepsilon -\alpha_{n}>0$ we get  $\varepsilon -\alpha_{n}> \abs{\norm{a_{i}^{n}-a_{i}} - \alpha_{n}} \geq \norm{a_{i}^{n}-a_{i}} - \alpha_{n}$. It follows that $ \{i\in \cI:\abs{\norm{a_{i}^{n}-a_{i}} -\alpha_{n}}<\delta\}\subset   \Omega_{n}(\varepsilon)$.
By directness, this implies that   $\Omega_{n}(\varepsilon)\in \cU$. 
\end{remark}

From now on we fix a faithful and non-degenerate representation of $\cA$ on $B(\mathcal{H})$. 

\begin{lemma}\label{existenciauwot}
 Let $\cU$ be an ultrafilter defined over $\cI$ and let $\cA$ be a $C^{\ast}$-algebra. For each $(a_{i})_{i\in\cI}\in\prod_{\cI}{\mathcal{A}}$, there exists a unique element $a_{\cU\text{-}WOT} \in B(\mathcal{H})$ 
  such that for every $\xi,\eta\in\mathcal{H}$, it holds that $\langle a_{\cU\text{-}WOT} \xi,\eta\rangle = \displaystyle\lim_{\cU}{\langle a_{i}\xi,\eta\rangle}$.
 The operator $a_{\cU\text{-}WOT}$ is called the $\cU$-WOT-limit of $(a_{i})_{i\in\cI}$. 

\end{lemma}

\begin{proof}
Let $(a_{i})_{i\in\cI}\in\prod_{\cI}{\mathcal{A}}$ and let $\xi,\eta \in\cH$. Then $(\langle a_{i}\xi,\eta \rangle)_{i\in\cI}\subset \C$ is bounded, hence it has a $\cU$-limit, which is denoted by  $b_{\xi,\eta}$. It is easy to see that the map $(\xi,\eta) \mapsto b_{\xi,\eta}$ is a bounded sesquilinear form on $\cH\times \cH$.
 Take $a_{\cU\text{-}WOT}\in B(\mathcal{H})$ the unique operator associated to it. 
\end{proof}

\begin{definition}
Let $\cU$ be an ultrafilter defined over $\cI$ and let $\cA$ be a $C^{\ast}$-algebra. An element $(a_{i})_{i\in\cI}\in\prod_{\cI}{\cA}$ is $\cU$-{\it strict} convergent if there exists an operator $a_{\cU}\in B(\cH)$ such that, for every $x\in \cA$, and every $\varepsilon>0$ we have $\{i\in \cI:\norm{a_{i}x-a_{\cU}x}<\varepsilon,\norm{xa_{i}-xa_{\cU}}<\varepsilon\}\in\cU$. The operator $a_{\cU}$ is called the $\cU$-{\it strict limit} of $(a_{i})_{i\in \cI}$. Observe that $a_{\cU}x$ and $xa_{\cU}$ are elements of $\cA$ for every $x\in\cA$.
\end{definition}

In what follows, it will be convenient to have the following notation at hand.

\begin{notation}\label{notacion}
Let $(a_{i})_{i\in \cI},(b_{i})_{i\in \cI}\in\displaystyle\Pi_{\cI}\cA$ that are $\cU$-strict convergent to $a_{\cU}$ and $b_{\cU}$, respectively. For every $x\in\cA$ and every $\varepsilon>0$,  put
$$A_{x}(\varepsilon) := \{i\in \cI:\norm{x(a_{i}-a_{\cU})},\norm{(a_{i}-a_{\cU})x}< \varepsilon\}\in \cU,$$
$$B_{x}(\varepsilon) := \{i\in \cI:\norm{x(b_{i}-b_{\cU})},\norm{(b_{i}-b_{\cU})x}< \varepsilon\}\in \cU.$$
\end{notation}

\begin{proposition}
Let $\cU$ be an  ultrafilter defined over $\cI$ and let $\cA$ be a $C^{\ast}$-algebra. If $(a_{i})_{i\in\cI},(b_{i})_{i\in\cI}\in\displaystyle\Pi_{\cI}\cA$ define the same element in $\cA^{\cU}$, then
\begin{enumerate}
\item if $(a_{i})_{i\in\cI}$ is $\cU$-WOT convergent to $a_{\cU\text{-}WOT}$, then $(b_{i})_{i\in\cI}$ is $\cU$\it{-WOT} convergent to $a_{\cU\text{-}WOT}$;\label{wduwotconv}
\item if $(a_{i})_{i\in\cI}$ is $\cU$-strict convergent to $a_{\cU}$, then $(b_{i})_{i\in\cI}$ is $\cU$\it{-strict} convergent to $a_{\cU}$.\label{wdustconv}
\end{enumerate}
\end{proposition}

\begin{proof}
 To prove \eqref{wduwotconv}, take $\xi,\eta\in\mathcal{H}$ of norm 1, let $\varepsilon>0$, and take
   $i$ in the set $\{i \in \cI:\norm{a_{i}-b_{i}}< \frac{\varepsilon}{2}\}\cap\{i\in\cI:\abs{\langle a_{\cU\text{-}WOT}\xi,\eta\rangle - \langle a_{i}\xi,\eta\rangle}<\frac{\varepsilon}{2}\}\in\cU$. 
  Then
\begin{equation*}
\abs{\langle a_{\cU\text{-}WOT}\xi,\eta\rangle - \langle b_{i}\xi,\eta\rangle} \leq \abs{\langle (a_{\cU\text{-}WOT} -a_{i})\xi,\eta\rangle} + \abs{\langle (a_{i}-b_{i})\xi,\eta\rangle}< \varepsilon.
 \end{equation*}
 It follows that the set $\{i\in\cI:\abs{\langle a_{\cU\text{-}WOT}\xi,\eta\rangle - \langle b_{i}\xi,\eta\rangle}<\varepsilon\}$ belongs to $\cU$.

To prove \eqref{wdustconv}, take $\varepsilon>0$, $x\in\cA$ and $i\in \{i\in \cI:\norm{a_{i}-b_{i}}< \frac{\varepsilon}{2\norm{x}}\}\cap A_{x}(\frac{\varepsilon}{2}) \in\cU$. Then
\begin{equation*}
\norm{x(b_{i}-a_{\cU})} \leq \norm{x(b_{i}-a_{i})} + \norm{x(a_{i}-a_{\cU})} < \varepsilon,
\end{equation*}
\begin{equation*}
\norm{(b_{i}-a_{\cU})x}\leq \norm{(b_{i}-a_{i})x} + \norm{(a_{i}-a_{\cU})x} < \varepsilon.
\end{equation*}
It follows that the set  $\{i\in \cI:\norm{x(b_{i}-a_{\cU})},\norm{(b_{i}-a_{\cU})x}< \varepsilon\} $ belongs to $\cU$. 
\end{proof}

\begin{proposition}\label{asucalgebra}
Let $\cU$ be an ultrafilter defined over $\cI$ and let $\cA$ be a $C^{\ast}$-algebra.   The set
$$\cA^{s \cU} :=\{(a_{i})_{i\in\cI} \in\cA^{\cU}:\text{there exists }a_{\cU}\in B(\mathcal{H}): (a_{i})_{i\in\cI} \text{ is }\cU\text{-strict convergent to } a_{\cU} \}$$
 is a $C^{\ast}$-algebra.
\end{proposition}

\begin{proof}
Let $(a_{i})_{\in\cI},(b_{i})_{i\in\cI}\in\cA^{s\cU}$ that are $\cU$-strict convergent to $a_{\cU}$ and $b_{\cU}$, respectively, let $\lambda\neq 0$ be a complex number, and let $x\in\cA$. By following  Notation \ref{notacion}, if $i\in A_{x}\left(\displaystyle\frac{\varepsilon}{2}\right)\cap B_{x}\left(\displaystyle\frac{\varepsilon}{2\abs{\lambda}}\right)\in\cU$, then
\begin{equation*}\norm{(a_{i}+\lambda b_{i} -a_{\cU} -\lambda b_{\cU})x} \leq \norm{(a_{i}-a_{\cU})x} + \abs{\lambda}\norm{(b_{i}-b_{\cU})x} < \varepsilon,\end{equation*}
and 
\begin{equation*}\norm{x(a_{i}+\lambda b_{i} -a_{\cU} -\lambda b_{\cU})} \leq \norm{x(a_{i}-a_{\cU})} + \abs{\lambda}\norm{x(b_{i}-b_{\cU})} < \varepsilon.\end{equation*}
It follows that 
$$A_{x}\left(\frac{\varepsilon}{2}\right)\cap B_{x}\left(\frac{\varepsilon}{2\abs{\lambda}}\right)\subset \{i\in \cI:\norm{(a_{i}+\lambda b_{i} -a_{\cU} -\lambda b_{\cU})x}<\varepsilon,\norm{x(a_{i}+\lambda b_{i} -a_{\cU} -\lambda b_{\cU})}<\varepsilon \},$$ 
so $(a_{i}+\lambda b_{i})_{i\in \cI}$ is $\cU$-strict convergent to $a_{\cU}+\lambda b_{\cU}$. 

It is clear that $\cA^{s \cU}$ is $*$-closed. To show that $\cA^{s \cU}$ is closed under taking products, set $M =\sup_{i\in\cI} \{\norm{a_{i}},\norm{b_{i}}\}$ and take $i\in A_{x}\big(\displaystyle\frac{\varepsilon}{2M}\big)\cap A_{b_{\cU}x}\big(\displaystyle\frac{\varepsilon}{2}\big)\cap B_{x}\big(\displaystyle\frac{\varepsilon}{2M}\big)\cap B_{xa_{\cU} }\big(\displaystyle\frac{\varepsilon}{2}\big)\in\cU$. Then
\begin{equation*}
\norm{(a_{i}b_{i} -a_{\cU}b_{\cU})x} \leq \norm{(a_{i}b_{i} - a_{i}b_{\cU})x} + \norm{(a_{i}b_{\cU} - a_{\cU}b_{\cU})x}
\leq \varepsilon
\end{equation*}
and
\begin{equation*}
\norm{x(a_{i}b_{i} -a_{\cU}b_{\cU})} \leq \norm{x(a_{i}b_{i} - a_{\cU} b_{i})} + \norm{x(a_{\cU} b_{i} - a_{\cU}b_{\cU})} 
\leq \varepsilon,
 \end{equation*}
which means that $(a_{i}b_{i})_{i\in\cI}$ is $\cU$-strict convergent to $a_{\cU}b_{\cU}$. 

It is left to show that $\cA^{s\cU}$ is norm closed. Let $((a_{i}^{n})_{i\in\cI})_{n\in\N}$ be a sequence in $\cA^{s\cU}$ that converges  to $(\alpha_{i})_{i\in\cI}$ in $\cA^{\cU}$. We need to see that $(\alpha_{i})_{i\in\cI}$ is $\cU$-strict convergent.

As a first step, we will show that, for a fixed element $x\in\cA$, $(\alpha_{i}x)_{i\in \cI}$ and $(x\alpha_{i})_{i\in \cI}$ have $\cU$-limit in $\cA$ (in the sense of Definition \ref{convergenciaultrafiltro}).

Let $x\in\cA$ be fixed and $x\neq 0$. For each $n\in\N$ let $a^{n}_{\cU}$ be the $\cU$-strict limit of $(a_{i}^{n})_{i\in \cI}\in\cA^{s\cU}$. We proceed to show that $(a^{n}_{\cU}x)_{n\in\N}$ is a Cauchy sequence in $\cA$. 
By  Remark \ref{conjuntosomega}, for every $\varepsilon>0$, there exists  $n_{0}\in\N$ such that, for all $n\geq n_0$, the sets  
$\Omega_{n}\big(\frac{\varepsilon}{4\norm{x}}\big)$ are elements of $\cU$. 
It follows that the  set
 \begin{equation*}
\{i\in \cI:\norm{a_{i}^{n}x-a^{n}_{\cU}x}<\frac{\varepsilon}{4}\} \cap \{i\in \cI:\norm{a_{i}^{m}x-a^{m}_{\cU}x}<\frac{\varepsilon}{4}\} \cap \Omega_{n}\
\big(\frac{\varepsilon}{4\norm{x}}\big)\cap \Omega_{m}\big(\frac{\varepsilon}{4\norm{x}}\big)
\end{equation*} 
is an element of $\cU$ for all $n,m\geq n_0$.  Take $i$ in this set. Then 
\begin{align*}
\norm{a^{n}_{\cU}x-a^{m}_{\cU}x} &\leq \norm{a^{n}_{\cU}x-a_{i}^{n}x} + \norm{a_{i}^{n}x-a_{i}^{m}x}+ \norm{a_{i}^{m}x-a^{m}_{\cU}x}\\
&\leq \frac{\varepsilon}{4} + \norm{a_{i}^{n} -\alpha_{i}}\norm{x} + \norm{\alpha_{i} - a_{i}^{m}}\norm{x} +\frac{\varepsilon}{4} < \varepsilon.
\end{align*}

Let  $\rho(x):=\displaystyle\lim_{n\in\N}{a^{n}_{\cU}x}\in\cA$. 
We will show that $\displaystyle\lim_{\cU}{\alpha_{i}x}=\rho(x) $, that is,\\  $\{i\in \cI:\norm{\rho(x) - \alpha_{i}x}<\varepsilon\}\in\cU$ for each $\varepsilon>0$. 
Let $n\in\N$ large enough  such that
 \begin{equation*}
 \norm{\rho(x) - a^{n}_{\cU}x}<\frac{\varepsilon}{3}\,\,\,\,\,\, \text{ and }\,\,\,\,\,
\Omega_n\left(\frac{\varepsilon}{3\norm{x}}\right)\in\cU.
\end{equation*}
For such $n\in\N$, take
 $i\in\{i\in \cI:\norm{a^{n}_{\cU}x-a_{i}^{n}x}<\displaystyle\frac{\varepsilon}{3}\}\cap \Omega_n\big(\displaystyle\frac{\varepsilon}{3\norm{x}}\big)\in\cU$.
 Then
 \begin{equation*}
\norm{\rho(x)-\alpha_{i}x} \leq \norm{\rho(x) - a^{n}_{\cU}x} + \norm{a^{n}_{\cU}x-a_{i}^{n}x} + \norm{a_{i}^{n}x-\alpha_{i}x} \leq\varepsilon.
\end{equation*}
Repeating this with  $(x\alpha_{i})_{i\in\cI}$ concludes the first step. 

 Let $\alpha_{\cU\text{-}WOT}\in B(\mathcal{H})$ be the $\cU$-WOT-limit of $(\alpha_{i})_{i\in\cI}$.  We will show that $ \alpha_{\cU\text{-}WOT}x=\rho(x)$. Take $\eta,\xi\in\mathcal{H}$ of norm 1, $\varepsilon>0$, and
$$i\in \{i\in \cI:\abs{\langle(\alpha_{i}-\alpha_{\cU\text{-}WOT})x\xi,\eta \rangle}<\frac{\varepsilon}{2}\}\cap \{i\in \cI:\norm{\rho(x) - \alpha_{i}x}<\frac{\varepsilon}{2}\}\in\cU.$$
We then have that
$$\abs{\langle(\rho(x)-\alpha_{\cU\text{-}WOT}x)\xi,\eta \rangle} \leq \abs{\langle(\rho(x)-\alpha_{i}x)\xi,\eta\rangle} + \abs{\langle(\alpha_{i}x-\alpha_{\cU\text{-}WOT}x)\xi,\eta \rangle} <\varepsilon,
$$
which implies that $\alpha_{\cU\text{-}WOT}x=\rho(x)= \displaystyle\lim_{\cU}{\alpha_{i}x}$.
Therefore, for all $x\in\cA$ and $\varepsilon>0$,
we have $\{i\in \cI:\norm{\alpha_{\cU\text{-}WOT}x-\alpha_{i}x}<\varepsilon\}\in\cU$.
In a similar manner, one shows that $\{i\in \cI:\norm{x\alpha_{\cU\text{-}WOT}-x\alpha_{i}}<\varepsilon\}\in\cU$.
It follows that $(\alpha_{i})_{i\in\cI}$ is $\cU$-strict convergent to $\alpha_{\cU\text{-}WOT}$.
\end{proof}

\begin{proposition}
Let $\cU$ be an ultrafilter defined over $\cI$ and let $\cA$ be a $C^{\ast}$-algebra. The set 
 \begin{equation*}\label{setJ}J:=\{(a_{i})_{i\in\cI}\in\cA^{s\cU}:a_{i} \text{ is }\cU\text{-strict convergent to } 0\}\end{equation*}
  is an ideal of $\cA^{s\cU}$.  
\end{proposition}

\begin{proof}

We only have to show that $J$ is norm closed. Consider $((a_{i}^{n})_{i\in\cI})_{n\in\N}\subset J$  a sequence that converges in norm to $(\alpha_{i})_{i\in\cI}\in  \cA^{s\cU}$. 
Let $\alpha_{\cU}$ be the $\cU$-strict limit of $(\alpha_{i})_{i\in\cI}$. 
Let $\varepsilon>0$. Take $n\in\N$ large enough such that 
 $\Omega_{n}\Big(\displaystyle\frac{\varepsilon}{3\norm{x}}\Big)\in \cU$, 
 and
$i\in\{i\in \cI:\norm{(\alpha_{i}-\alpha_{\cU})x}<\frac{\varepsilon}{3}\}\cap\{i\in \cI:\norm{a_{i}^{n}x}<\frac{\varepsilon}{3}\}\cap\Omega_{n}\Big(\displaystyle\frac{\varepsilon}{3\norm{x}}\Big) \in \cU.$
  We then have that
  $\norm{\alpha_{\cU} x}\leq \norm{(\alpha_{\cU}-\alpha_{i})x} + \norm{\alpha_{i}x} \leq \frac{\varepsilon}{3} + \norm{(\alpha_{i}-a_{i}^{n})x} + \norm{a_{i}^{n}x}\leq\varepsilon.$
Since the action of $\cA$ on $\mathcal{H}$ is nondegenerate, $\alpha_{\cU} = 0$. 
\end{proof}

There exists a natural embedding of $\cA$ in $\cA^{s\cU}$, via the constant sequences $a\mapsto (a)_{i\in \cI}$. It is clear that this element is $\cU$-strict convergent to $a$. Moreover, since the representation of $\cA$ in $B(\cH)$ is faithful and nondegenerate, it follows that there exists a natural embedding of $\cA$ in $\cA^{s\cU}/J$.

Recall that an ideal $I$ in a $C^{\ast}$-algebra $\cA$ is  {\it essential} if $I\cap K$ is nontrivial for every ideal $K\neq\{0\}$, or equivalently, $aI=0$ implies $a=0$.

\begin{lemma}\label{essential ideal} Let $\cU$ be an ultrafilter defined over $\cI$ and let $\cA$ be a $C^{\ast}$-algebra.  Consider the $C^{\ast}$-algebra $\cA^{s\cU}/J$. 
The image of $\cA$ in $\cA^{s\cU}/J$ is an essential ideal. 
\end{lemma} 

\begin{proof}
 Take $(b_{i})_{i\in \cI}\in\cA^{s\cU}$, let $b_{\cU}$ be its $\cU$-strict limit, and take $a\in \cA$. Then $(b_{i}a)_{i\in \cI}$  and $(b_{\cU}a)_{i\in\cI}$ are both 
 $\cU$-strict convergent to $b_{\cU}a$.
It follows that $(b_ia)_{i\in\cI}$ and $(b_{\cU}a)_{i\in\cI}$ are equal in $\cA^{s\cU}/J$. Analogously,
 $(ab_i)_{i\in\cI}$ is equal to $(ab_{\cU})_{i\in\cI}$ in $\cA^{s\cU}/J$.

 Suppose that $J'\subset \cA^{s\cU} /J$ is an ideal such that $J'\cap\cA=\{0\}$. 
If $(b_{i})_{i\in \cI}\in\cA^{s\cU} $  projects to $J'$, then $(b_{i}x)_{i\in\cI}\in J$,for each $x\in\cA$. Let $b_{\cU}$ be the $\cU$-strict limit of 
$(b_{i})_{i\in\cI}$.  
Hence $(b_{i}x)_{i\in\cI}$ is $\cU$-strict convergent to $b_{\cU}x$. Then 
$b_{\cU}x=0$ for all $x\in\cA$. It follows that $b_{\cU}=0$ and then $J'=\{0\}$.
\end{proof}

\subsection{Ultrafilters and approximate units}

Every $C^{\ast}$-algebra $\cA$ has an approximate unit, namely, there exist a directed set $\cI$ and a net $(e_{i})_{i\in \cI}\subset \cA$ such that for every $x\in\cA$, the nets $(xe_i)_{i\in \cI}$ and $(e_ix)_{i\in \cI}$ converge to $x$ (see  \cite[Chapter I.4]{MR3839621}). 
Approximate units can be taken to be positive and uniformly bounded, in which case they are elements of $\prod_{\cI}{\cA}$.
In what follows, we will focus in the case where the ultrafilters are defined over this directed set $\cI$. Observe that
for a cofinal ultrafilter $\cU$,  the sets $\{i\in \cI:\norm{xe_{i}-x}<\varepsilon\}$ and $\{i\in \cI: \norm{e_{i}x-x}<\varepsilon\}$
belong to $\cU$, for every $x\in \cA$ and for every $\varepsilon>0$.
Moreover, when $\cA$ is  unital,  $\cI$ can be taken to be the set with one element $\{1_{\cA}\}$ and the approximate unit to be equal to $\{1_{\cA}\}$. In this case, the only ultrafilter is the set $\{1_{\cA}\}$, which is cofinal.

\begin{theorem}\label{main theorem}
Let $\cA$ be a $C^{\ast}$-algebra and let $(e_{i})_{i\in \cI}\in \prod_{i\in\cI}\cA$ be an approximate unit for $\cA$. Let $\cU$ be  a cofinal ultrafilter over $\cI$. Then the $C^{\ast}$-algebra $\cA^{s\cU}/J$ is the multiplier algebra of $\cA$.
\end{theorem}

\begin{proof}
We saw in Lemma \ref{essential ideal} that $\cA$ is an essential ideal in  $\cA^{s\cU}/J$. 
We are left to show that 
for any  $C^{\ast}$-algebra $\mathcal{B}$ containing $\cA$ as an ideal,
 there is a unique $C^{\ast}$-homomorphism $\varphi:\mathcal{B}\to \cA^{s\cU}/J$ 
 such that $\varphi (a) = a$.\\
To this end, let $b\in\mathcal{B}$ and consider $\psi:\mathcal{B}\to\cA^{s\cU}$ defined by $\psi(b)=(be_{i})_{i\in\cI}$. 
To see that $\psi$ is well defined, let $\varepsilon>0$ and let $x\in\cA$. 
Let $i_{0}\in\cI$ such that if $i\geq i_{0}$, then $\norm{xbe_{i} - xb}<\varepsilon$. 
Let $i_{1}\in\cI$ such that if $i\geq i_{1}$, then $\norm{e_{i}x-x}<\displaystyle\frac{\varepsilon}{\norm{b}}$.
By the cofinality of $\cU$ one obtains that $\{i\in\cI:\norm{be_{i}x - bx}<\varepsilon,\norm{xbe_{i} - xb}<\varepsilon\}\in\cU$. 
Observe that since $b\notin B(\cH)$, the last line does not imply that $(be_{i})_{i\in\cI}\in \cA^{s\cU}$.
We must ``represent'' $b$ in $B(\cH)$. 
To this end, 
let $b_{\cU\text{-}WOT}\in B(\cH)$ be the $\cU$-WOT-limit of $(be_{i})_{i\in\cI}\in \cA^{\cU}$, an argument similar to one given in the proof of Proposition \ref{asucalgebra}, shows that $b_{\cU\text{-}WOT}x = bx$ and
 $xb_{\cU\text{-}WOT} = xb$. 
 Thus $(be_{i})_{i\in\cI}$ is $\cU$-strict convergent to $b_{\cU\text{-}WOT}$. The same procedure shows that $(e_{i}b)_{i\in\cI}$ is $\cU$-strict convergent to $b_{\cU\text{-}WOT}$.

Call $\pi$ the quotient projection to $\cA^{s\cU}/J$, and let  $\varphi = \pi \circ \psi$. 
It is clear that $\varphi$ is  linear and bounded.
Since $\psi(b^{*}) = (b^{*}e_{i})_{i\in\cI}$ and $\psi(b)^{*} = (e_{i}b^{*})_{i\in\cI}$ and they are both $\cU$-strict convergent 
to $(b^{*})_{\cU\text{-}WOT}$, hence 
$\psi(b^{*}) - \psi(b)^{*}$ is an element of $J$, 
so $\varphi $ is a $\ast$-preserving homomorphism.

To see that $\varphi$ is multiplicative, fix $b,b'\in\mathcal{B}$ of norm $1$ and
take $x\in\cA$, $\varepsilon>0$, $M = \sup_{i\in\cI}\{\norm{e_{i}}\}$ and $i$ in the set
$$
 \{i\in\cI:\norm{e_{i}x-x}<\frac{\varepsilon}{3M}\}\cap\{i\in\cI:\norm{x-e_{i}x}<\frac{\varepsilon}{3M}\}\cap \{i\in\cI:\norm{e_{i}b'x-b'x}<\frac{\varepsilon}{3M}\},$$
which is an element of $\cU$. 
Then
\begin{align*}
\norm{(be_{i}b'e_{i} - bb'e_{i})x}  &\leq 
\norm{e_{i}b'e_{i}x - e_{i}b'x} + \norm{e_{i}b'x-b'x} 
+ \norm{b'x-b'e_{i}x} \\
&\leq \norm{e_{i}b'}\norm{e_{i}x-x} + \norm{e_{i}b'x-b'x} + \norm{b'}\norm{x-e_{i}x} < \varepsilon.
\end{align*}
Take $i\in\{i\in\cI:\norm{xbe_{i} - xb}<\displaystyle\frac{\varepsilon}{M}\}\in\cU$.
 Then
$\norm{x(be_{i}b'e_{i} - bb'e_{i})} \leq \norm{xbe_{i} - xb}\norm{b'e_{i}} <\varepsilon.$
It follows that $\psi(b)\psi(b') - \psi(bb')=(be_{i}b'e_{i} - bb'e_{i})_{i\in\cI}$  is an element of  $J$.

Since $(ae_{i} -a)_{i\in\cI}$ is $\cU$-strict convergent to $0$,  for all $a\in\cA$, so $\varphi(a) = a$ in $ \cA^{s\cU}/J$.

Suppose that there exists another $C^{\ast}$-homomorphism $\varphi':\mathcal{B}\to\cA^{s\cU}/J$ such that $\varphi'(a) = a$ for all $a\in\cA$. 
Then \begin{equation*}\varphi'(b)a = \varphi'(b)\varphi'(a) = ba =\varphi(b)\varphi(a) = \varphi(b)a.\end{equation*}
By Lemma \ref{essential ideal}, $\varphi(b)=\varphi'(b)$.
\end{proof}

Ultraproducts provide a new point of view for dealing with multiplier algebras. For instance,
the identification of $\eme(\cA)$ with $\cA^{s\cU}/J$  yields an easy proof of the next characterization of multipliers, 
without using double centralizers.

\begin{corollary} \label{corolary}
$\eme(\cA)$ is isomorphic to $\eme:=\{m\in  B(\cH): \text{ for all } a\in \cA, am\in \cA\,, \, ma\in \cA\}$. In particular, $\eme(\cA)$ is unital.
\end{corollary}
\begin{proof}
Consider  $\varphi:\cA^{s\cU}\to \eme$ defined by $\varphi((a_{i})_{i\in\cI}) = \displaystyle\lim_{\cU\text{-} strict}{a_{i}}$. This map is well defined, it is a 
$C^{\ast}$-homomorphism (Proposition \ref{asucalgebra}), and $\ker(\varphi)=J$.
To show that $\varphi$ is surjective, let $m\in\mathcal{M}$. Then $am\in\cA$ and  $ma\in\cA$  for every $a\in\cA$. 
Hence, for all $\varepsilon>0$, the sets $\{i\in\cI:  \norm{a(me_{i} - m)}<\varepsilon\}$
and  $\{i\in\cI: \norm{(me_{i}-m)a}<\varepsilon\}$ are elements of $\cU$. 
Then  $(me_{i})_{i\in \cI}\in\cA^{\cU}$ is $\cU$-strict convergent to $m$.

Taking $m=1$, it follows that the image of $(e_{i})_{i\in \cI}$ in $\cA^{s\cU}/J$ is the unit of $\cA^{s\cU}/J$.
\end{proof}

For a second application,  observe that  every $C^{\ast}$-homomorphism $\phi:\cA\rightarrow \mathcal{B}$ defines a natural $C^{\ast}$-homomorphism $\phi':\cA^{\cU}\rightarrow \mathcal{B}^{\cU}$. When $\phi$ is surjective, a proof similar to one given in Proposition \ref{asucalgebra} shows that $\phi'(\cA^{s\cU})\subset\mathcal{B}^{s\cU}$. This together with Theorem \ref{main theorem} immediately gives the following known result.
\begin{proposition}
Let $\cA, \mathcal{B}$ be $C^{\ast}$-algebras and let $\phi:\cA \rightarrow \mathcal{B}$ be a surjective homomorphism. The natural extension
 $\phi':\cA^{\cU}\rightarrow \mathcal{B}^{\cU}$ induces the following commutative diagram:
\begin{equation*}
  \begin{tikzcd}
    \cA \arrow{r} \arrow{d}{\phi} & \cA^{s\cU} \arrow{r} \arrow{d}{\phi'} & \eme(\cA) \arrow{r} \arrow{d}{\phi''} & \eme(\cA)/\cA \arrow{d}{\phi''{'}} \\
      \mathcal{B} \arrow{r}& \mathcal{B}^{s\cU} \arrow{r}  & \eme(\mathcal{B}) \arrow{r} & \eme(\mathcal{B})/\mathcal{B}.
  \end{tikzcd}
\end{equation*}

\end{proposition}

\section{the case of commutative and separable $C^{\ast}$-algebras}\label{sec3}
 
  Recall that when a $C^{\ast}$-algebra $\cA$ is separable, it is $\sigma$-unital, namely, there exists a countable approximate unit 
  (see  \cite[Chapter I.4]{MR3839621}). 
   That entails that the index set $\cI$ of the previous section can be taken to be equal to $\N$, in which case nonprincipal ultrafilters are cofinal and  $\prod_{\cI} \cA$ is $\ell^{\infty}(\cA)$.
 
In \cite{2016arXiv161009276A}, the authors built the multiplier algebra for  commutative and separable $C^{\ast}$-algebras using ultraproducts of $C^{\ast}$-algebras. More precisely they considered $\cA = C_{0}(X)$ where $X$ is a second countable, locally compact topological space and took  $\cU$  a nonprincipal ultrafilter defined over $\N$ to construct the multiplier algebra of $\cA$,  $C_{b}(X)$, identifying it with a quotient of a sub-$C^{\ast}$-algebra of $\prod_{\cI}{\cA}$.  For that, the authors usee the key fact that the hypothesis on $X$ entails the existence of a proper metric compatible  \cite{MR1563581}. In what follows, we will then identify the second countable, locally compact topological space $X$ with the metric space $(X,d)$, where $d$ is a proper metric on $X$. The closed ball of radius $r>0$ centered at a fixed base-point $o\in X$, will be denoted by $B_{o}(r)$. The main technical tool of \cite{2016arXiv161009276A} is the following definition. 

\begin{definition} \cite[Section 3]{2016arXiv161009276A}\label{uequi}
Let $(X,d)$ be as in the preceding discussion.
  Let $\cA = C_{0}(X)$ and let $\cU$ be a nonprincipal ultrafilter over $\mathbb{N}$. For $(f_{n})_{n\in\N}\in\ell^{\infty}(\cA)$, we say that $ (f_{n})_{n\in\N}$ is {\it $\cU$-equicontinuous on bounded sets} if, for every $r,\varepsilon>0$, there is  $\delta>0$ such that the set $\{n\in\mathbb{N}:\text{ for all }s,t\in B_{o}(r)\text{ with } d(s,t)<\delta \implies \abs{f_{n}(s)-f_{n}(t)}<\varepsilon\}$ belongs to $\cU$.
\end{definition}
 
Given $(f_{n})_{n\in\N}\in\ell^{\infty}(\cA)$, and  a fixed $x\in X$, the $\cU$-limit of the sequence $(f_{n}(x))_{n\in\N}$ is well defined. We denote $f_{\cU}:X\rightarrow\mathbb{C}$, $f_{\cU}(x):= \displaystyle\lim_{\cU}(f_{n}(x))$. The following fact was observed  in \cite{2016arXiv161009276A}.

\begin{lemma}\label{3.1 AG}
If $(f_{n})_{n\in\N}\in\ell^{\infty}(\cA)$ is $\cU$-equicontinuous on bounded sets, then $f_{\cU}$ is uniformly continuous on bounded sets.
\end{lemma}

The next proposition shows that our notion of $\cU$-strict convergence coincides with the notion of $\cU$-equicontinuity on bounded sets in the case of $\cA = C_{0}(X)$, where $X$ is a locally compact, second countable, topological space. This entails that the work done here in section \ref{sec2} is indeed a generalization of the work done in \cite[Section 3]{2016arXiv161009276A}.
\begin{proposition}
 Take $(f_{n})\in\ell ^{\infty}(\cA)$ and let $f_{\cU}(x)=\displaystyle\lim_{\cU}(f_{n}(x))$. The following two conditions are equivalent:
\begin{enumerate}
\item The sequence $(f_{n})_{n\in\N}$ is $\cU$-equicontinuous on bounded sets.\label{goldequi}
\item The sequence $(f_{n})_{n\in\N}$ is $\cU$-strict convergent to $f_{\cU}$.\label{strictnatural}
\end{enumerate}
\end{proposition}

\begin{proof}

To show that $\eqref{goldequi}$ implies $\eqref{strictnatural}$,
let $(f_{n})_{n\in\N}\in\ell ^{\infty}(\cA)$  be $\cU$-equicontinuous on bounded sets. 
Let $\varepsilon>0$ and let $g\in C_{0}(X)$. Set $M = \sup_{n\in\N}\{\norm{f_{n}},\norm{g}\}$, and take $K\subset X$ a compact set such that $\abs{g(x)}<\displaystyle\frac{\varepsilon}{2M}$ if $x\notin K$. 
There exists $\delta_{1}$ such that for $x,y\in K$ with $d(x,y)<\delta_{1}$,  $\{n\in\N: \abs{f_{n}(x)-f_{n}(y)}\leq \displaystyle\frac{\varepsilon}{3M}\}\in\cU$.
By Lemma \ref{3.1 AG},  there exists $\delta_{2} > 0$ such that $\abs{f_{\cU}(x)-f_{\cU}(y)}<\displaystyle\frac{\varepsilon}{3}$,
for $x,y\in K$ with $d(x,y)<\delta_{2}$.
Take $\delta = \min\{\delta_{1},\delta_{2}\}$ 
and cover $K$ with a finite number of balls $B_{x_j}(\delta)$ of radius $\delta$ centered at $x_{j}$, $j=1,\ldots,m$.
 Since $f_{\cU}(x_j) = \displaystyle\lim_{\cU}f_{n}(x_j)$, 
it follows that the sets $A_{j} =\left\{n\in\N:\abs{f_{n}(x_j)-f_{\cU}(x_j)}<\displaystyle\frac{\varepsilon}{3M}\right\}$ belong to $\cU$.
 Therefore if $n\in \bigcap _{j=1}^{m} A_{j}\in \cU $, and $x\in K$,  taking $x_{j}$ with $d(x,x_{j})<\delta$ we get
\begin{multline*}\abs{(f_{n}(x) - f_{\cU}(x))g(x)} \leq \abs{f_{n}(x) - f_{n}(x_{j})}\norm{g} + \abs{f_{n}(x_{j}) - f_{\cU}(x_{j})}\norm{g} +\\
 \abs{f_{\cU}(x_{j}) - f_{\cU}(x)}\norm{g} < \varepsilon.\end{multline*}
On the other hand, if $x\notin K$, then $\abs{(f_{n} - f_{\cU})g (x)}  < \varepsilon$. 
This shows that  $\{n\in \N: \|f_ng-f_{\cU}g\|<\varepsilon\}\in\cU$.

To show the converse, 
take $g\in C_{0}(X)$ such that $g= 1$ in $B_{o}(r)$. By hypothesis, for all $\varepsilon>0$
 the sets $\{n\in\N:\norm{(f_{n} - f_{\cU})g}<\varepsilon\}$ are in $\cU$ and are infinite. So we can build a subsequence $(f_{n_{k}})_{k\in\N}$ such that $(f_{n_{k}})_{k\in\N}$ is uniformly convergent to $f_{\cU}$ in $B_{0}(r)$.
  By hypothesis, closed balls are compact, so $f_{\cU}$ is uniformly continuous  on $B_{o}(r)$.
 Let $\delta>0$ such that $d(x,y)<\delta$ implies $\abs{f_{\cU}(y) - f_{\cU}(x)}< \varepsilon$ in $B_{o}(r)$.
Take $n\in\{n\in\N:\norm{(f_{n} - f_{\cU})g}<\displaystyle\frac{\varepsilon}{3}\}\in\cU$.
Then for $x,y\in B_{o}(r)$ such that $d(x,y)<\delta$ we have
\begin{equation*}
\abs{f_{n}(x)-f_{n}(y)}\leq \abs{f_{n}(x)-f_{\cU}(x)} + \abs{f_{\cU}(x) - f_{\cU}(y)} + \abs{f_{n}(y)-f_{\cU}(y)} \leq \varepsilon,
\end{equation*}
therefore $\{n\in\N:\norm{(f_{n} - f_{\cU})g}<\frac{\varepsilon}{3}\}$ is a subset of $\{ n\in\N: \text { for all } x,y\in B_{o}(r) \text{ with } d(x,y)<\delta \implies \abs{f_{n}(x)-f_{n}(y)}<\varepsilon\}$. This  shows that the last set belongs to $\cU$.
\end{proof}

\section{An application to boundary amenability of groups}\label{sec4}

Let $\Gamma$ be a countable group. Let $\cA$ be a unital $C^{\ast}$-algebra endowed with a $\Gamma$-action by $\ast$-automorphisms. 
Let $C_{c}(\Gamma,\cA)$ be the space of finitely supported functions from $\Gamma$ to $\cA$.
This is a $\ast$-algebra with the product given by
\begin{equation*}
T\ast S (\gamma) = \sum_{\gamma_{1}\gamma_{2}=\gamma}{T(\gamma_{1})(\gamma_{1}\cdot S(\gamma_{2}))}
\end{equation*}
and the involution given by
\begin{equation*}
T^{*} (\gamma) = \gamma \cdot T(\gamma^{-1})^{*}.
\end{equation*}
The space $C_{c}(\Gamma,\cA)$ has a pre-Hilbert $\cA$-module structure via the inner product $\langle T,S\rangle_{\cA} := \sum\limits_{\gamma\in\Gamma}{T(\gamma)^{*}S(\gamma)}$, and the corresponding norm $\norm{T}_{\cA} := \langle T,T \rangle^{1/2}$.

\begin{definition}\label{exactitud}
A group $\Gamma$ is boundary amenable (or exact) if there exist a compact space $Y$ and a sequence $(S_{i})_{i\in\N}\subset C_{c}(\Gamma,C(Y))$ such that 
\begin{enumerate}
\item\label{item1} $S_{i}$ is positive, that is, $S_{i}(\gamma)\geq 0$ for all $ \gamma\in\Gamma$;
\item\label{item2}  $\sum_{\gamma\in\Gamma}{S_{i}^{2}(\gamma)} = 1$;
\item\label{item3}  for every $\gamma_{1}\in\Gamma$, it follows that $\norm{S_{i} - \gamma_{1}\ast S_{i}}$ goes to $0$ when $i$ goes to infinity.
\end{enumerate}
\end{definition}

In \cite[Section 4]{2016arXiv161009276A}, the authors applied their ultraproduct construction of the multiplier algebra of a separable $C^{\ast}$ algebra to show that groups acting properly and transitively on a locally finite tree are boundary amenable. The assumption that the action is proper and transitive implies that the vertex stabilizers are all isomorphic finite groups of cardinal $m$. In the proof given in \cite{2016arXiv161009276A}, this constant $m$ contains all the information that is needed about the stabilizers. 
 Since it is known that groups acting on a locally finite tree with exact stabilizers are boundary amenable (see, for instance,\cite{MR1978888,MR2391387,MR2243738,MR1871980}), it is interesting to see if the strategy developed in \cite{2016arXiv161009276A} can be extended to prove this more general result. We partially achieve this in the proof of the next theorem.
\begin{theorem}\label{teorema}
If a  countable group $\Gamma$  acts transitively on a locally finite tree $\cT$ with boundary amenable  stabilizers, 
 then $\Gamma$ is boundary amenable.
\end{theorem}
\begin{proof}
Fix a vertex $o\in\cT$ as a base-point.  We will denote with $B(i)$ the closed ball of radius $r$ centered in $o$.
The geodesic that connects $o$ with $t$ will be denoted by $[o,t]$. In what follows  $Stab\{o\}$ will be denoted with $\Lambda$.

Since the action of $\Gamma$ on $\cT$ is transitive, we choose a cross section for it,
 namely, for each $v\in \cT$ we choose $g_v\in\Gamma$ such that $g_{v}o=v$.

 Set  $\big(\Gamma_i\big)_{i\in\N}$ an increasing  sequence of finite subsets of $\Gamma$ that covers $\Gamma$ and let
$$\Lambda_{i} := \bigcup\limits_{v\in B(i)}\bigcup\limits_{w\in B(i)}{\left(g_{v}^{-1}\Gamma_{i} g_{w}\cap\Lambda\right)}.$$ 
Since $\cT$ is locally finite,  $\Lambda_{i}$ is  finite.
   
Since $\Lambda$ is boundary amenable, there exist a compact set $Y$
 and a sequence
$(S_i)_{i\in\N}\subset C_c((\Lambda,C(Y))$ satisfying the three conditions of Definition \ref{exactitud}.
The compact space $Y$ will be taken to be a $\Gamma$-space (see \cite[pg. 178]{MR2391387}).
Moreover, we can assume that $S_{i}$ is defined over all $\Gamma$
 by extending it by $0$ in $\Gamma\setminus\Lambda$.
 For every $i\in\mathbb{N}$, let 
\begin{equation*}
\kappa(i):=\min\{k\in\N: ||S_{l}-\lambda\ast S_{l}\|<\frac{1}{i} \text { for all } l\geq k \text{ and for all }\lambda \in \Lambda_{i}\}.
\end{equation*}
Consider the $C^{\ast}$-algebra $\cA:=C_{0}(\cT\times Y)$. Since $Y$ is compact and $\cT$ is countable,
 $\cA$ is a $\sigma$-unital $C^{\ast}$-algebra, hence we are in the case discussed at the beginning of Section \ref{sec3}. 
 Let $\cU$ be a non principal ultrafilter over $\N$ and consider the corona algebra
 $\eme(\cA)/\cA$, which we identify with $\cA^{s\cU}/J\big/\cA$. This is a unital $C^{\ast}$-algebra.

For each $i\in\N$, set $T_{i}:\Gamma\rightarrow \cA^{s\cU}$, 
where $T_{i}(\gamma)=(T_{i}^{n}(\gamma))_{n\in\N}$ 
is defined by
\begin{equation*}\label{folnerseq}
T_{i}^{n}(\gamma)(t,y) = \frac{1}{\sqrt{\Abs{[o,t]\cap B(i)}}}\chi_{B(n)}(t)\sum_{v\in B(i)}{\chi_{[o,t]}(v)S_{\kappa(i)}(g_{v}^{-1}\gamma)(g_{v}^{-1} y) },
\end{equation*}
where $\Abs{[o,t]\cap B(i)}$ denotes the number of points in the geodesic 
$[o,t]$ that lie inside the ball $B(i)$.
Note that $(T_{i}^{n}(\gamma))_{n\in\N}$ is $\cU$-strict convergent to
 $$T_{i} = T_{i}(\gamma)(t,y) = \frac{1}{\sqrt{\Abs{[o,t]\cap B(i)}}}\sum_{v\in B(i)}{\chi_{[o,t]}(v)S_{\kappa(i)}(g_{v}^{-1}\gamma)(g_{v}^{-1} y) }$$
We will  denote by $T _{i}$ its projection  to $\cA^{s\cU}/J\big/\cA.$ We remark that the definition of $T_i$ is inspired by the definition  of 
$\mu_{x,y}$ given by Ozawa in \cite{MR2243738}.\\
\noindent
 {\bf Claim:} The sequence $(T_{i})_{i\in\N}\subseteq C(\Gamma;\cA^{s\cU}/J\big/\cA)$ satisfies the conditions of Definition \ref{exactitud}.\\
 To show this, first observe that if $\Omega_{\kappa(i)}\subseteq \Lambda$ denotes the support of $S_{\kappa(i)}$, 
 then the support of $T_{i}$ is contained in 
$\bigcup_{v\in B(i)}g_{v}\Omega_{\kappa(i)}$.
Since $\cT$ is locally finite, this is a finite set.\\
To show that condition (\ref{item1}) holds true, it is enough to notice that $T_{i}$  is sum and product of positive functions. \\
To show that condition (\ref{item2}) holds true, note that for each fixed $\gamma\in\Gamma$, there exists only one $g_v$ such that $g_{v}^{-1}\gamma\in\Lambda$. 
Then there exists at most one $g_{v}$ such that $S_{\kappa(i)}(g_{v}^{-1}\gamma)$ is nonzero. 
Then 
\begin{align*}
\sum_{\gamma\in\Gamma}(T_{i})^{2}(t,y) &= 
\sum_{\gamma\in\Gamma}
\frac{1}{\Abs{[o,t]\cap B(i)}}\sum_{v\in B(i)}{\chi_{[o,t]}(v)S_{\kappa(i)}^{2}(g_{v}^{-1}\gamma)(g_{v}^{-1} y) }\\
&=\frac{1}{\Abs{[o,t]\cap B(i)}}\sum_{v\in B(i)}\chi_{[o,t]}(v)\sum_{\gamma\in\Gamma}{S_{\kappa(i)}^{2}(g_{v}^{-1}\gamma)(g_{v}^{-1} y) } \\
&=\frac{1}{\Abs{[o,t]\cap B(i)}}\sum_{v\in B(i)}\chi_{[o,t]}(v)\sum_{\lambda\in\Lambda}{S_{\kappa(i)}^{2}(\lambda)(g_{v}^{-1} y) } \\
&=\frac{1}{\Abs{[o,t]\cap B(i)}}\sum_{v\in B(i)}\chi_{[o,t]}(v) = 1.
\end{align*}

It remains to show that condition (\ref{item3}) holds true. To this end, fix $\gamma_{1}\in\Gamma$.   
Observe that if $(t,y)\in \cT\times Y$, then 
$\gamma_{1}\ast T_{i}(\gamma)(t,y) = T_{i}(\gamma_{1}^{-1}\gamma)(\gamma_{1}^{-1}t,\gamma_{1}^{-1}y)$. 
Hence 
\begin{align}
\label{calculo de norma}
\norm{T_{i} - \gamma_{1}\ast T_{i}}_{C_{c}(\Gamma,\cA^{s\cU}/J/\cA)}^{2} 
&= 
\norm{\sum\limits_{\gamma\in\Gamma}{\left(T_{i}(\gamma) - \gamma_{1}\ast T_{i}(\gamma) \right)^{2}}}_{\cA^{s\cU}/J\big/\cA}
\nonumber \\
&=
\inf_{a\in\cA}{\|\sum\limits_{\gamma\in\Gamma}{\left(T_{i}(\gamma) - \gamma_{1}\ast T_{i}(\gamma) \right)^{2} -a\|}_{\cA^{s\cU}/J}}
\nonumber\\
&=
\inf_{a\in\cA}{\|2- 2\sum_{\gamma\in\Gamma}{T_{i}(\gamma)(\gamma_{1}\ast T_{i}(\gamma))- a\|}_{\cA^{s\cU}/J}}
 \end{align}

If we set $\theta(i,t):=\Abs{[o,t]\cap B(i)}^{-1/2}\,\Abs{[o,\gamma_{1}^{-1}t]\cap B(i)}^{-1/2}$,
then 
\begin{align}\label{suma larga}
&\sum_{\gamma\in\Gamma}{T_{i}(\gamma)(t,y)T_{i}(\gamma_{1}^{-1}\gamma)(\gamma_{1}^{-1}t,\gamma_{1}^{-1}y)} 
\nonumber\\
&=\theta(i,t)\sum_{\substack{v\in B(i)\\ w\in B(i)}}{\chi_{[o,t]}(v)\chi_{[o,\gamma_{1}^{-1}t]}(w)} 
\sum_{\gamma\in\Gamma}{S_{\kappa(i)}(g_{v}^{-1}\gamma)(g_{v}^{-1} y)
S_{\kappa(i)}(g_{w}^{-1}\gamma_{1}^{-1}\gamma)(g_{w}^{-1}\gamma_{1}^{-1} y)}
\nonumber\\
&=\theta(i,t)\sum_{\substack{v\in B(i)\\ w\in B(i)}}{\chi_{[o,t]}(v)\chi_{[o,\gamma_{1}^{-1}t]}(w)} 
\sum_{\lambda\in\Lambda}{S_{\kappa(i)}(\lambda)(z)S_{\kappa(i)}(g_{w}^{-1}
\gamma_{1}^{-1}g_{v}\lambda)(g_{w}^{-1}\gamma_{1}^{-1}g_{v} z)},
\end{align}
where the last equality follows from the changes of variables 
$\lambda:=g_v^{-1}\gamma$ and $z:= g_{v}^{-1} y$, 
and by recalling that $S_{\kappa(i)}(\lambda) = 0$ if $\lambda \notin \Lambda$.

In order to get that 
$S_{\kappa(i)}(g_{w}^{-1}\gamma_{1}^{-1}g_{v}\lambda) \neq 0$, 
it is necessary that 
$g_{w}^{-1}\gamma_{1}^{-1}g_{v}$
belongs to $\Lambda$. 
Note that for a fixed $v$, there exists only one $g_w$ such that 
$g_{w}^{-1}\gamma_{1}^{-1}g_{v}\in\Lambda$. For this $g_w$,
 we have that $w= \gamma_1^{-1}v$ and 
 $\chi_{[o,\gamma_{1}^{-1}t]}(w)=\chi_{[o,\gamma_{1}^{-1}t]}
  (\gamma_{1}^{-1}v) = \chi_{[\gamma_{1}o,t]}(v)$.
Therefore, \eqref{suma larga} is equal to the following sum, only depending on  $v$.
\begin{equation*}
\theta(i,t)\sum_{v\in B(i)\cap \gamma_{1}B(i)}{\chi_{[o,t]}(v)\chi_{[\gamma_{1}o,t]}(v)} \sum_{\lambda\in\Lambda}{S_{\kappa(i)}(\lambda)(z)S_{\kappa(i)}(g_{w}^{-1}\gamma_{1}^{-1}g_{v}\lambda)(g_{w}^{-1}\gamma_{1}^{-1}g_{v} z)}.
\end{equation*}
Replacing this in  \eqref{calculo de norma}, we get that
$\norm{T_{i} - \gamma_{1}\ast T_{i}}_{C_{c}(\Gamma,\cA^{s\cU}/J/\cA)}^{2} $
is equal to
\begin{multline}
\label{igualdad 2}
\inf_{a\in\cA}\sup_{(t,y)\in T\times Y}
\Bigg\{\Abs{2-a(t,y)-\\
2\theta(i,t)\sum_{v\in B(i)\cap \gamma_{1}B(i)}{\chi_{[o,t]}(v)\chi_{[\gamma_{1}o,t]}(v)} \sum_{\lambda\in\Lambda}{S_{\kappa(i)}(\lambda)(z)S_{\kappa(i)}(g_{w}^{-1}\gamma_{1}^{-1}g_{v}\lambda)(g_{w}^{-1}\gamma_{1}^{-1}g_{v} z)}}\Bigg\}.
\end{multline}
By the triangle inequality, to estimate \eqref{igualdad 2}, it is enough to estimate the following two quantities
\begin{equation}
2\theta(i,t)\sum_{v\in B(i)\cap \gamma_{1} B(i)}{\chi_{[o,t]}(v)\chi_{[\gamma_{1}o,t]}(v)
\Abs{ 1 -\sum_{\lambda\in\Lambda}{S_{\kappa(i)}(\lambda)(z)S_{\kappa(i)}(g_{w}^{-1}
\gamma_{1}^{-1}g_{v} \lambda)(g_{w}^{-1}\gamma_{1}^{-1}g_{v} z)}}}
\label{eq1}
\end{equation}
and
\begin{equation}
\Abs{2 -2\theta(i,t)\sum_{v\in B(i)\cap \gamma_{1} B(i)}{\chi_{[o,t]}(v)\chi_{[\gamma_{1}o,t]}(v)-a(t,y)}}\label{eq2}
\end{equation}

In order to bound \eqref{eq1}, first note that 
$$2\Abs{1 -\sum_{\lambda\in\Lambda}{S_{\kappa(i)}
(\lambda)(z)S_{\kappa(i)}(g_{w}^{-1}\gamma_{1}^{-1}g_{v} \lambda)(g_{w}^{-1}\gamma_{1}^{-1}g_{v} z)} }
 = \norm{S_{\kappa(i)} - (g_{v}^{-1}\gamma_{1}g_{w})\ast S_{\kappa(i)}}^{2}.$$
Then note that there exists $r\in\N$ such that $\gamma_1\in\Gamma_i$ for all $i\geq r$.
It follows that for all $i\geq r$, $g_v^{-1}\gamma_1g_w\in \Lambda_i$, whenever $v,w\in B_{i}$. 
Then, by the definition of $\kappa(i)$, it follows that
\begin{equation*}
\norm{S_{\kappa(i)} - (g_{v}^{-1}\gamma_{1}g_{w})\ast S_{\kappa(i)}}^{2} \leq \frac{1}{i}, \text{ for all } i\geq r ,
 \text{ whenever } v\in B(i)\cap \gamma_{1} B(i). 
\end{equation*}
Then, for all $i\geq r$,  \eqref{eq1} is  bounded by 
\begin{align*}
 \frac{1}{i} \theta(i,t)  &\sum_{v\in B(i)\cap \gamma_{1} B(i)}{\chi_{[o,t]}(v)\chi_{[\gamma_{1}o,t]}(v)} \\
 &\leq  \frac{1}{i} \theta(i,t) \Big(\sum_{v\in B(i)\cap \gamma_{1} B(i)}{\chi_{[o,t]}(v)}\Big)^{1/2}
 \Big(\sum_{v\in B(i)\cap \gamma_{1} B(i)}\chi_{[\gamma_{1}o,t]}(v)\Big)^{1/2} \\
 &\leq  \frac{1}{i}\theta(i,t) \Abs{[o,t]\cap B(i)}^{1/2}\Abs{[\gamma_1o,t]\cap \gamma_1B(i)}^{1/2}\\
 &=  \frac{1}{i}\theta(i,t) \Abs{[o,t]\cap B(i)}^{1/2}\Abs{[o,\gamma_1^{-1}t]\cap B(i)}^{1/2}\\
 &=   \frac{1}{i} .
\end{align*}
In order to bound \eqref{eq2}, for each $i\in\mathbb{N}$,  we choose
$$a(t,y)=a(t):=
\left(2 -2\theta(i,t)\sum_{v\in B(i)\cap \gamma_{1} B(i)}
{\chi_{[o,t]}(v)\chi_{[\gamma_{1}o,t]}(v)}\right) \chi_{B(i)\cup\gamma_1B(i)}(t).$$

This choice of $a\in\cA$ left us to bound \eqref{eq2} only when $t\notin B(i)\cup\gamma_1B(i)$. In this case
$\theta(i,t)=\frac{1}{i}$. 
 Moreover, if $i\geq d(o,\gamma_1o)$ 
 then $$\Abs{[o,t]\cap[\gamma_1o,t]\cap B(i)\cap \gamma_1B(i)}\geq i-d(o,\gamma_1o).$$
It follows that in \eqref{eq2}  we have the bound 
\begin{align*}
2-\frac{2}{i}\sum_{v\in B(i)\cap \gamma_{1} B(i)}{\chi_{[o,t]}(v)\chi_{[\gamma_{1}o,t]}(v)}
&=2-\frac{2}{i}\Abs{[o,t]\cap[\gamma_1o,t]\cap B(i)\cap \gamma_1B(i)}\\
&\leq 2-\frac{2}{i}(i-d(o,\gamma_1o))
\\
&
=\frac{d(o,\gamma_1o)}{i}.
\end{align*}
All this combined entails that if
 $\varepsilon>0$,
 and $i\geq\max\{r;d(o,\gamma_1o),2/\varepsilon; 2d(o,\gamma_1o)/\varepsilon\}$, then
$\norm{T_{i} - \gamma_{1}\ast T_{i}}_{C_{c}(\Gamma,\cA^{s\cU}/J/\cA)}^{2} <\varepsilon$.
\end{proof}

\textbf{Acknowledgments.} F. Poggi was supported in part by a CONICET Doctoral Fellowship. R. Sasyk was supported in part through the grant PIP-CONICET 11220130100073CO. We thank Prof. Isaac Goldbring for his comments on the first version of the article that helped us improve the exposition.
We thank the referee for his detailed reading of the manuscript and for prompting us to add a section on boundary amenability of groups.

\bibliography{Bibliografia}
	\bibliographystyle{plain}
\end{document}